\documentclass[11pt]{article}
\usepackage{amssymb,amsfonts,amsmath,amsthm}%
\usepackage[dvips]{graphicx,color}%
\newcommand{\figdraft}{false}%
\newcommand{\figfile}[1]{#1}%
\theoremstyle{plain}%
\newtheorem{theorem}{Theorem}[]%
\newtheorem{lemma}[theorem]{Lemma}%
\newtheorem{assumption}[theorem]{Assumption}%
\newtheorem{remark}[theorem]{Remark}%
\definecolor{colorGreen}{rgb}{0.,0.67,0}
\definecolor{colorRed}{rgb}{0.67,0.,0}
\definecolor{colorBlue}{rgb}{0.,0.,0.67}

\newcommand{\ie}{{i.e.}}

\newcommand{\Len}{L}

\newcommand{\iu}{\mathtt{i}}
\newcommand{\mhexp}[1]{{{\mathtt{e}}^{#1}}}

\newcommand{\sgn}{\mathrm{sgn}}

\newcommand{\fspace}[1]{{\mathsf{#1}}}
\newcommand{\fspaceL}{\fspace{L}}
\newcommand{\fspaceC}{\fspace{C}}
\newcommand{\fspaceW}{\fspace{W}}

\newcommand{\ol}[1]{{\overline{#1}}}

\newcommand{\phase}{{\varphi}}

\newcommand{\Rset}{{\mathbb{R}}}

\newcommand{\Nset}{{\mathbb{N}}}

\newcommand{\oointerval}[2]{(#1,\,#2)}%
\newcommand{\ccinterval}[2]{[#1,\,#2]}%

\newcommand{\sh}{{\rm sh}}

\newcommand{\tdots}{{...}}%

\newlength{\mhpicDwidth}
\newlength{\mhpicDvsep}
\newlength{\mhpicDhsep}

\newlength{\mhpicPwidth}
\newlength{\mhpicPvsep}
\newlength{\mhpicPhsep}

\newlength{\mhpicWhsep}

\newcommand{\pair}[2]{{\left({#1},\,{#2}\right)}}
\newcommand{\skp}[2]{{\left\langle{#1},\,{#2}\right\rangle}}

\newcommand{\at}[1]{{\left({#1}\right)}}
\newcommand{\nat}[1]{(#1)}
\newcommand{\bat}[1]{{\big(#1\big)}}
\newcommand{\Bat}[1]{{\Big(#1\Big)}}

\newcommand{\triple}[3]{{\left({#1},\,{#2},\,{#3}\right)}}

\newcommand{\ul}[1]{\underline{#1}}

\newcommand{\bigpar}{\par\quad\newline\noindent}

\newcommand{\wt}[1]{{\widetilde{#1}}}
\newcommand{\wh}[1]{{\widehat{#1}}}

\newcommand{\jump}[1]{{|\![#1]\!|}}
\newcommand{\mean}[1]{{\langle#1\rangle}}

\newcommand{\norm}[1]{\|{#1}\|}
\newcommand{\abs}[1]{\left|{#1}\right|}

\newcommand{\dint}[1]{\,\mathrm{d}#1}

\newcommand{\Ga}{{\Gamma}}

\newcommand{\eps}{{\varepsilon}}

\newcommand{\la}{{\lambda}}
\newcommand{\si}{{\sigma}}

\newcommand{\MiTime}{t}

\newcommand{\calA}{\mathcal{A}}

\newcommand{\calC}{\mathcal{C}}

\newcommand{\calH}{\mathcal{H}}

\newcommand{\calL}{\mathcal{L}}
\newcommand{\calM}{\mathcal{M}}
\newcommand{\calN}{\mathcal{N}}

\newcommand{\calP}{\mathcal{P}}

\newcommand{\calT}{\mathcal{T}}

\usepackage[%
a4paper,%
nohead,%
twoside,
ignorefoot,
scale=0.8,
bindingoffset=.4cm
]{geometry}%
\usepackage{float}%
\begin{document}%
%
%
\title{Action minimizing fronts in general FPU-type chains }%
\date{\today}%
\author{%
Michael Herrmann%
\thanks{
    Oxford Centre for Nonlinear PDE (OxPDE),
    {\tt{michael.herrmann@maths.ox.ac.uk}}
}%
}%
\maketitle
%
\begin{abstract}%
We study atomic chains with nonlinear nearest neighbour interactions and prove the existence
of fronts (heteroclinic travelling waves with constant asymptotic states). Generalizing
recent results of Herrmann and Rademacher we allow for non-convex interaction potentials and
find fronts with non-monotone profile. These fronts minimize an action integral and can only
exists if the asymptotic states fulfil the macroscopic constraints and if the
interaction potential satisfies a geometric graph condition. Finally, we illustrate our
findings by numerical simulations.
\end{abstract}%
%
%
\quad\newline\noindent%
\begin{minipage}[t]{0.15\textwidth}%
Keywords: %
\end{minipage}%
\begin{minipage}[t]{0.8\textwidth}%
\emph{Fermi-Pasta-Ulam chain}, %
\emph{heteroclinic travelling waves}, \\%
\emph{conservative shocks}, %
\emph{least action principle} %
\end{minipage}%
\medskip
\newline\noindent
\begin{minipage}[t]{0.15\textwidth}%
MSC (2000): %
\end{minipage}%
\begin{minipage}[t]{0.8\textwidth}%
37K60, 
47J30, 
70F45, 
74J30 
\end{minipage}%
%
%
%
%
\setcounter{tocdepth}{5} %
\setcounter{secnumdepth}{4}
{\scriptsize{\tableofcontents}}%

%
%
%
\section{Introduction}
%
Nonlinear Hamiltonian lattices like chains of interacting atoms or coupled oscillators are
ubiquitous in mathematics, physics, and material sciences. The most famous example, and most
elementary model for a crystal, is a chain of identical atoms that interact by nearest
neighbour forces. In reminiscence of the pioneering paper by Fermi, Pasta, and Ulam
\cite{FPU55} one usually refers to such systems as FPU or FPU-type chains.
\par
Although FPU chains are quite simple lattice models they exhibit a rich and complicate
dynamical behaviour, and we still lack a complete understanding of their dynamical
properties. A mayor topic in the analysis of FPU chains is therefore the investigation of
coherent structures such as travelling waves and breathers. Travelling waves are highly
symmetric, exact solutions to the underlying lattice equation. They can be regarded as the
fundamental modes of nonlinear wave propagation and provide much insight into the energy
transport in discrete media. In this paper we aim in contributing to the general theory by
studying fronts, i.e., heteroclinic travelling waves that connect two different constant
states.
\bigpar
The dynamics of FPU chains is governed by the lattice equation
\begin{align}%
\label{Eqn:Intro.FPU1}
\ddot{x}_j=\Phi^\prime\at{x_{j+1}-x_{j}}-\Phi^\prime\at{x_{j}-x_{j-1}}.
\end{align}
Here $x_j=x_j\at{t}$ denotes the position of the $j^{th}$ atom at time $t$, $\Phi$ is the
interaction potential, and the atomic mass is normalized to $1$. Introducing the atomic
distances $r_j=x_{j+1}-x_{j}$ and velocities $v_{j}=\dot{x}_{j}$ we can reformulate
\eqref{Eqn:Intro.FPU1} as
\begin{align}%
\label{e:FPU}
\dot{r}_j = v_{j+1}-v_j\;,\qquad \dot{v}_j =
\Phi^{\prime}\at{r_j} - \Phi^{\prime}\at{r_{j-1}}.
\end{align}
A travelling wave is a special solution to \eqref{e:FPU} that satisfies the ansatz
\begin{align}
\label{e:TW.Ansatz}
r_j\at\MiTime=R\at{j-\si{t}},\qquad
v_j\at\MiTime=V\at{j-\si{t}}.
\end{align}
$R$ and $V$ are the \emph{profile functions} for distances and velocities, $\si$ denotes the
wave speed, and $\phase=j-\si{t}$ is the \emph{phase variable}. In dependence of the
properties of $R$ and $V$ travelling waves come in different types. \emph{Wave trains} have
periodic profiles and are investigated in \cite{FV99,PP00,DHM06}. They describe oscillatory
solutions to \eqref{Eqn:Intro.FPU1} and provide the building blocks for Whitham's modulation
theory. Another important class of travelling waves are \emph{solitons} (or \emph{solitary
waves}), where $R$ and $V$ are localized over a constant background state. The existence of
solitons in lattices is a nontrivial problem and has been studied intensively during the last
20 years. We refer to \cite{FW94,SW97,FM02,Pan05,SZ07,Her10a} for variational methods, and to
\cite{Ioo00,IJ05} for an approach via spatial dynamics and centre manifold reduction.
\bigpar
In this paper we study \emph{fronts} which have heteroclinic shape and satisfy
\begin{align}%
\label{e:AsymptoticStates} %
\lim\limits_{\phase\to\pm\infty}R\at\phase=r_\pm
,\qquad%
\lim\limits_{\phase\to\pm\infty}V\at\phase=v_\pm
\end{align}%
with $\pair{r_-}{v_-}\neq\pair{r_+}{v_+}$. Fronts have attracted much less interest than
solitons, maybe because they only exist if the asymptotic states satisfy some very
restrictive conditions. In particular, $\Phi^\prime$ must have at least one turning point
between $r_-$ and $r_+$, and this excludes for instance the famous Toda potential.
Nonetheless, fronts in FPU chains appear naturally in atomistic Riemann problems, see
\cite{HR10a} for numerical simulations, and are important in the context of phase
transitions.
\par
The first rigorous result about fronts we are aware of is the bifurcation criterion from
\cite{Ioo00}. It implies that fronts with small jumps between the asymptotic states exist
only if $\Phi^\prime$ has a convex-concave turning point. Recently, the existence of fronts
was proven by variational methods in \cite{HR10b}. The existence theorem therein does not
require the asymptotic states to be close to each other but is restricted to convex
potentials $\Phi$. The proof relies on a Lagrangian action integral for fronts with
prescribed asymptotic states and uses the direct approach to establish the existence of
minimizers. A similar approach is used in \cite{KZ09a,KZ09b} to prove the existence of
fronts for sine-Gordon chains.
\par
In this paper we generalize the method from \cite{HR10b} and prove the existence of fronts
without convexity assumption on $\Phi$. Our main result can be summarized as follows.
\begin{theorem}
\label{Intro:MainTheo}
Action minimizing front solutions to \eqref{e:FPU} exist under the following hypotheses:
\begin{enumerate}
\item[$\at{i}$]
The asymptotic states and the front speed satisfy the \emph{macroscopic constraints},
which take the form of three independent jump conditions.
\item[$\at{ii}$]
The potential satisfies the \emph{graph condition} with respect to the asymptotic states.
\item[$\at{iii}$]
Some technical assumptions are also satisfied.
\end{enumerate}
Moreover, there is no front without $\at{i}$,
and no action minimizing front without $\at{ii}$.
\end{theorem}
The assumptions in Theorem \ref{Intro:MainTheo} will be specified below.  The macroscopic
constraints, see Lemma \ref{Lem:JumpCond}, are algebraic relations and link fronts to energy
conserving shocks of the p-system, which is the na\"{\i}ve continuum limit of FPU chains. In
particular, they determine the wave speed $\si$ and imply that the asymptotic strains $r_-$
and $r_+$ cannot be chosen independently of each other. The graph condition reformulates the area condition
from \cite{HR10b} and requires that the graph of $\Phi$ is below the shock parabola
associated with the asymptotic states. Both the macroscopic constraints and the graph
condition appear naturally in our variational existence proof and guarantee that the action
integral is well-defined and bounded from below.
\bigpar
Closely related to fronts are heteroclinic waves with oscillatory tails. These are travelling
wave solutions to \eqref{e:FPU} which approach two different periodic waves for
$\phase\to\pm\infty$. Such oscillatory fronts are used to describe martensitic phase
transitions and to derive kinetic relations in solids \cite{BCS01a,BCS01b,AP07,Vai10}. The
only available existence results, however, concern piecewise quadratic potentials, which
allow for simplifying the travelling wave equation by means of Fourier transform, see
\cite{TV05,SCC05,SZ09}. It remains a challenging problem for future research to give
alternative, maybe variational, existence proofs that cover more general chains.
\bigpar
The paper is organized as follows. In \S\ref{sec:waves} we discuss the macroscopic
constraints and normalize the asymptotic states. Moreover, we reformulate the front equation
as an eigenvalue problem for a nonlinear integral operator. In \S\ref{sec:proof} we set the
existence problem into a variational framework and characterize fronts as minimizers of an
action integral. Or main technical result is Theorem \ref{Theo:Minimiser} and guarantees that
this action integral attains its minimum on a suitable set of candidates for fronts. The
proof uses \emph{separations of phases}, which are introduced in \S\ref{sec:proof:sop} and
allow to extract convergent subsequences from action minimizing sequences. Finally, we
present some numerical simulations in \S\ref{sec:num}.
%
\section{Preliminaries about fronts}\label{sec:waves}
%
%
Substituting the travelling wave ansatz \eqref{e:TW.Ansatz} into \eqref{e:FPU} yields
\begin{align}%
\label{e:tw} %
\si\tfrac{\dint}{\dint\phase}R(\phase)+V(\phase+1)-V(\phase)=0,\qquad
\si\tfrac{\dint}{\dint\phase}V(\phase)+ \Phi^\prime\bat{R\at{\phase}}-
\Phi^\prime\bat{R\at{\phase-1}}=0,
\end{align}
which is a nonlinear system of advance-delay-differential equations. Moreover, combining both
equations we readily verify the energy law
\begin{align}%
\label{e:tw.energy} %
\si\tfrac{\dint}{\dint\phase}\Bat{\tfrac{1}{2}V^2\at\phase+\Phi\at{R\at\phase}}+
\Phi^\prime\bat{R\at\phase}V\at{\phase+1}-
\Phi^\prime\bat{R\at{\phase-1}}V\at{\phase}=0.
\end{align}
%
\subsection{Macroscopic constraints for the asymptotic states}%
%
We now derive the macroscopic constraints that couple the front speed $\si$ to the asymptotic
states $\pair{r_\pm}{v_\pm}$ from \eqref{e:AsymptoticStates}. To this end
we consider continuous observables
$\psi=\psi\pair{r}{v}$ and
denote by
\begin{align*}
\jump{\psi\pair{r}{v}}:=\psi\pair{r_+}{v_+}-\psi\pair{r_-}{v_-}
\qquad\text{and}\qquad
\mean{\psi\pair{r}{v}}:=\tfrac{1}{2}\at{\psi\pair{r_-}{v_-}+\psi\pair{r_+}{v_+}},
\end{align*}
the \emph{jump} and \emph{mean value}, respectively.
\par
The following result was proven in \cite{HR10b} (see also \cite{AP07}) by integrating
\eqref{e:tw} and \eqref{e:tw.energy} over a finite interval $\ccinterval{-N}{N}$ and passing
to the limit $N\to\infty$.  
\begin{lemma}%
\label{Lem:JumpCond}
The asymptotic states of each front satisfy
\begin{align}
\label{Lem:JumpCond.Eqn1}
\si\jump{r}+\jump{v}=0,\qquad
\si\jump{v}+\jump{\Phi^\prime\at{r}}=0,\qquad
\si\jump{\tfrac{1}{2}v^2+\Phi\at{r}}+\jump{\Phi^\prime\at{r}v}=0.
\end{align}
\end{lemma}%
Heuristically, Lemma \ref{Lem:JumpCond} reflects that fronts transform into shock waves when 
passing to large spatial and temporal scales. The jump conditions \eqref{Lem:JumpCond.Eqn1} 
precisely mean that the asymptotic states correspond to an \emph{energy conserving shock} for 
the p-system and imply that each front satisfies mass, momentum, and energy. The p-system is the 
na\"{\i}ve continuum limit of FPU chains under the hyperbolic scaling and reads
\begin{align}
\label{Eqn:PSystem}
\partial_\tau{r}=\partial_y{v},\qquad
\partial_\tau{v}=\partial_y{\Phi^\prime\at{r}},
\end{align}
where $\tau=\eps{t}$ and $y=\eps{j}$ denote the macroscopic time and space, respectively, and
$\eps>0$ is a small scaling parameter. The conservation laws in \eqref{Eqn:PSystem}
correspond to mass and momentum, and imply the conservation of energy for smooth solutions,
that is
\begin{align}
\label{Eqn:PSystem.E}
\partial_\tau\at{\tfrac{1}{2}v^2+\Phi\at{r}}=\partial_\tau\at{v\Phi^\prime\at{r}}.
\end{align}%
The jump conditions for \eqref{Eqn:PSystem.E}, however, is independent of the jump conditions
for \eqref{Eqn:PSystem}. More details about the p-system and energy conserving shocks can be
found in \cite{HR10a,HR10b}.
\par
Using the discrete Leibniz rule %
\begin{math} %
\jump{\psi_1 \psi_2}=\jump{\psi_1}\mean{\psi_2}+\mean{\psi_1}\jump{\psi_2}
\end{math} %
we readily verify that \eqref{Lem:JumpCond.Eqn1} implies
\begin{align}
\label{Lem:JumpCond.Eqn2}
\jump{\Phi\at{r}}=\jump{r}\mean{\Phi^\prime\at{r}}
,\qquad%
\si^2=\jump{\Phi^\prime\at{r}}/\jump{r}.
\end{align}
Conversely, for any $\pair{r_-}{r_+}$ with \eqref{Lem:JumpCond.Eqn2}$_1$ there exist -- up to
Galilean transformations -- exactly two solutions to \eqref{Lem:JumpCond.Eqn1} which differ
in $\sgn{\si}$. We now characterize the geometric meaning of \eqref{Lem:JumpCond.Eqn2} and
refer to Figure \ref{fig:shocks} for an illustration.
\begin{figure}[ht!]
\centering{%
\includegraphics[width=0.975\textwidth, draft=\figdraft]%
{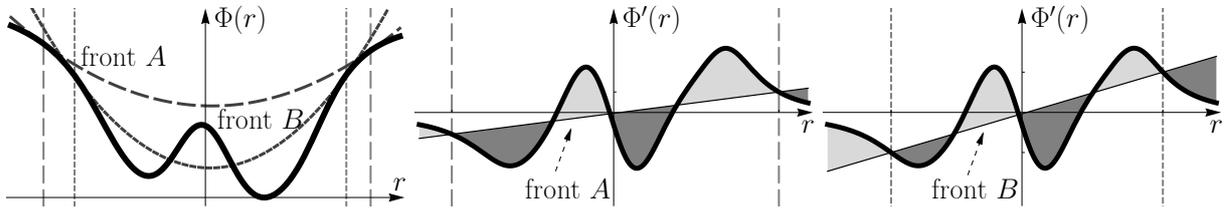}%
\caption{%
To each front there exists a parabola that touches the graph of $\Phi$ in both $r_-$ and
$r_+$. Consequently, the signed area between the graph of $\Phi^\prime$ and the secant
connecting $r_-$ and $r_+$ vanishes the stripe $\ccinterval{r_-}{r_+}$.
}%
\label{fig:shocks}
}%
\end{figure}
\begin{lemma}
The following conditions are equivalent:
\begin{enumerate}
\item[$\at{i}$]
$\triple{\si}{r_-}{r_+}$ fulfils \eqref{Lem:JumpCond.Eqn2},
\item[$\at{ii}$]
there exists a parabola that touches the graph of $\Phi$ in both
$r_-$ and $r_+$,
\item[$\at{iii}$]
the signed area between the graph of $\Phi^\prime$ and the secant connecting $r_-$ to
$r_+$ sums up to zero in $\ccinterval{r_-}{r_+}$.
\end{enumerate}
Moreover, each condition implies that $\Phi^\prime$ has at least one turning point between
$r_-$ and $r_+$.
\end{lemma}
\begin{proof}
Consider the parabola $f\at{r}=\tfrac{1}{2}ar^2+br+c$. The touching conditions
\begin{align*}
f\at{r_\pm}=\Phi\at{r_\pm},\quad
f^\prime\at{r_\pm}=\Phi^\prime\at{r_\pm}
\end{align*}
are equivalent to
\begin{align*}
\tfrac{1}{2}a\jump{r^2}+b\jump{r}=\jump{\Phi\at{r}},\quad
\tfrac{1}{2}a\mean{r^2}+b\mean{r}+c=\mean{\Phi\at{r}},\quad
\jump{r}a=\jump{\Phi^\prime\at{r}}
,\quad
a\mean{r}+b=\mean{\Phi^\prime\at{r}},
\end{align*}
and by $\tfrac{1}{2}\jump{r^2}=\mean{r}\jump{r}$ we conclude that $\at{i}$ and $\at{ii}$ are
equivalent via
\begin{align*}
a=\si^2
,\quad
b=\mean{\Phi^\prime\at{r}}-\si^2\mean{r}
,\quad
{c}=&
\mean{\Phi\at{r}}-\mean{r}\mean{\Phi^\prime\at{r}}+
\si^2\bat{\mean{r}^2-\tfrac{1}{2}\mean{r^2}}.
\end{align*}
The equivalence of $\at{ii}$ and $\at{iii}$ is immediate since the secant has slope $\si^2$,
and $\Phi^\prime$  must have a turning point because otherwise the graph of $\Phi^\prime$
would be either below or above the secant.
\end{proof}
Condition \eqref{Lem:JumpCond.Eqn2}$_1$ is the \emph{kinetic relation} for fronts and reveals
that the asymptotic states cannot be chosen arbitrarily. More precisely, for given $r_-$ and
$r_+$ we can choose $\si$ and $\jump{v}$ such that the first two jump conditions in
\eqref{Lem:JumpCond.Eqn1} (which correspond to mass and momentum) are satisfied. However, for
the energy condition \eqref{Lem:JumpCond.Eqn1}$_3$ to hold, $r_-$ and $r_+$ must additionally
fulfil \eqref{Lem:JumpCond.Eqn2}$_1$.  Form this we conclude that fronts do not exist if
$\Phi^\prime$ is either convex or concave, and that in general we cannot prescribe both $r_-$
and $r_+$.
\bigpar
We emphasize that \eqref{Lem:JumpCond.Eqn1} is in general not sufficient for the existence of
fronts, \ie, there exist energy conserving shocks in the p-system that can not be realized by
a front in FPU. In fact,  it was proven in \cite{Ioo00} that fronts bifurcate from
convex-concave but not from concave-convex turning points of $\Phi^\prime$. This disproves
the existence of \emph{subsonic} fronts with small jump heights although there exist the
corresponding energy conserving shocks.
\par
In order to prove the existence of action minimizing fronts we shall additionally to
\eqref{Lem:JumpCond.Eqn1} require that the graph of $\Phi$ is below the parabola defined by
the asymptotic states, see Assumption \ref{Ass:Pot}. In particular, our existence result
provides a front for Example $A$ from Figure \ref{fig:shocks} but does not cover Example $B$,
see Remark \ref{Rem:UnboundedL} and the examples in \S\ref{sec:num}.
%
%
%
\subsection{Normalization and reformulation}%
%
For our analysis in \S\ref{sec:proof} it is convenient to normalize the asymptotic states and
to reformulate the front equation \eqref{e:tw} as an eigenvalue problem for a nonlinear
integral operator.
\begin{lemma}
\label{Lem:Normalisation}
Up to affine transformations we can assume that
\begin{align}
\label{Norm.Problem.States}
\si=1
,\quad%
r_\pm=\pm1
,\quad%
v_\pm=\mp1
,\quad%
\Phi^\prime\at{\pm1}=\pm1
,\quad%
\Phi\at{\pm1}=\tfrac{1}{2}.
\end{align}
Moreover, with \eqref{Norm.Problem.States} the front equation is equivalent to
\begin{align}%
\label{Norm.Problem.Eqn}%
W=\calA{\Phi}^\prime\at{\calA{W}}
,\qquad%
\at{\calA{W}}\at\phase=
\int\limits_{\phase-\tfrac{1}{2}}^{\phase-\tfrac{1}{2}}W\at{\tilde\phase}\dint\tilde\phase,
\end{align}
where $W$ is a normalized profile with $\lim_{\phase\to\pm\infty}W\at\phase=\pm1$.
\end{lemma}
\begin{proof}
Let $U$ and $W$ be two normalized profiles such that
\begin{align*}
{R}\at\phase=\mean{r}+\tfrac{1}{2}\jump{r}U\at{\phase+1/2}
,\qquad%
V\at\phase=\mean{v}+\tfrac{1}{2}\jump{v}W\at\phase.
\end{align*}
Using the first two jump conditions from \eqref{Lem:JumpCond.Eqn1} we readily verify that
\eqref{e:tw} transforms into
\begin{align}%
\label{Lem:Normalisation.Eqn1}%
\tfrac{\dint}{\dint\phase}U\at\phase=W\at{\phase+1/2}-W\at{\phase-1/2}
,\qquad%
\tfrac{\dint}{\dint\phase}W\at\phase=\wh{\Phi}^\prime\bat{U\at{\phase+1/2}}-
\wh{\Phi}^\prime\bat{U\at{\phase-1/2}},
\end{align}
where the normalized potential
\begin{align*}
\wh{\Phi}\at{u}=\frac{4}{\jump{\Phi^\prime\at{r}}\jump{r}}
{\Phi}\Bat{\mean{r}+\tfrac{1}{2}\jump{r}\,u}
-\frac{2\mean{\Phi^\prime\at{r}}}{\jump{\Phi^\prime\at{r}}}\,{u}+\frac{1}{2}-
\frac{4\mean{\Phi\at{r}}}{\jump{\Phi^\prime\at{r}}\jump{r}}
\end{align*}
satisfies $\wh{\Phi}^\prime\at{\pm1}=\pm1$. Moreover, we have
$\wh{\Phi}\at{-1}=\wh{\Phi}\at{+1}=\tfrac{1}{2}$ if and only if the third jump condition
\eqref{Lem:JumpCond.Eqn1}$_3$ is satisfied. Towards \eqref{Norm.Problem.Eqn} now suppose
\eqref{Norm.Problem.States}. Integrating \eqref{Lem:Normalisation.Eqn1}$_1$ we find
$U=\calA{W}$, where the constant of integration vanishes due to $U\at{\pm1}=W\at{\pm1}=\pm1$,
and similarly we derive $W=\calA\wh{\Phi}^\prime\at{U}$ from
\eqref{Lem:Normalisation.Eqn1}$_2$.
\end{proof}
The front parabola for normalized data \eqref{Norm.Problem.States} is $r\mapsto\tfrac{1}{2}r^2$ 
and each solution to \eqref{Norm.Problem.Eqn} can be viewed as a perturbation of the \emph{shock
profile}
\begin{align}
\label{E:DefShock}%
W_\sh\at\phase=\sgn\phase=\left\{
\begin{array}{rcl}
+1&\text{for}&\phase<0,\\
0&\text{for}&\phase=0,\\
-1&\text{for}&\phase>0.
\end{array}
\right.
\end{align}
Notice that the residual of $W_\sh$, that is $W_\sh-\calA\Phi^\prime\at{\calA{W_\sh}}$, has
compact support.
\bigpar
We proceed with some preliminary remarks about the action of a front. Heuristically, the
action density in the normalized setting is given by
\begin{align*}
\tfrac{1}{2}W^2-\Phi\at{\calA{W}}
=\tfrac{1}{2}W^2-\tfrac{1}{2}\at{\calA{W}}^2+\Psi\at{\calA{W}}
\end{align*}
with
\begin{align}
\label{E:DefPsi}%
\Psi\at{r}=\tfrac{1}{2}r^2-\Phi\at{r},
\end{align}
so the action integral formally reads
\begin{align}
\label{E:DefAction2}%
\wt{\calL}\at{W}&= \int\limits_\Rset
\tfrac{1}{2}W^2-\tfrac{1}{2}\at{\calA{W}}^2+\Psi\at{\calA{W}} \dint\phase.
\end{align}
Notice that $\Psi$ is just the difference between the front parabola and $\Phi$, see Figure
\ref{fig:shocks}, and that $\wt{\calL}$ is well defined as long as $W$ approaches its
asymptotic states sufficiently fast. A further possibility for defining the action integral
was introduced in \cite{HR10b} for monotone $W$ and relies on the relative action integral
\begin{align*}
\wh{\calL}\at{W}&=
\int\limits_\Rset %
\Bat{\tfrac{1}{2}W^2-
\Phi\at{\calA{W}}}-
\Bat{
\tfrac{1}{2}W_\sh^2-
\Phi\at{\calA{W_\sh}}}
\dint\phase.
\end{align*}
Both approaches are linked by $\wh{\calL}\at{W}=\wt{\calL}\at{W}-\wt{\calL}\at{W_\sh}$ and
the symmetry of $\calA$, compare Lemma \ref{Lem:AProps}, formally implies
\begin{align*}
\partial\wh{\calL}\at{W}=\partial\wt{\calL}\at{W}=W-\calA\Phi^\prime\at{\calA{W}}.
\end{align*}
In \S\ref{sec:proof} we give a slightly different definition of $\wt\calL$, see
\eqref{Eqn:DefM} and \eqref{Eqn:DefL}, and establish the existence of minimizers.
%
%
\section{Existence of fronts}\label{sec:proof}
%
In this section we assume that the asymptotic states and the potential are normalized by
\eqref{Norm.Problem.States} and show that the fixed point equation \eqref{Norm.Problem.Eqn}
has a solution in some appropriate function space.
%
%
\subsection{Assumptions}%
\newcommand{\cond}[1]{(#1)}
We rely on the following standing assumptions on the function $\Psi$ from \eqref{E:DefPsi}.
Examples and counterexamples are given in \S\ref{sec:num}.
\begin{assumption}
\label{Ass:Pot}%
$\Psi$ is continuously differentiable and satisfies the following conditions:
\begin{enumerate}
\item[\cond{G}]
\emph{graph condition}: $\Psi\at{u}\geq\Psi\at{\pm1}=0$ for all $u\in\Rset$,
\item[\cond{X}]
\emph{genericity}:
$\Psi^{\prime\prime}\at{\pm1}>0$ and $\Psi\at{u}>0$ for all $u\neq\pm1$,
\item[\cond{M}]
\emph{monotone asymptotic behaviour}: $\Psi\at{u}$ is decreasing for $u\ll-1$ and
increasing for $u\gg1$.

\end{enumerate}%
\end{assumption}
Condition \cond{G} has a natural interpretation in terms of $\Phi$ and can easily be
reformulated for non-normalized data: It precisely means that the front parabola touches the
graph of $\Phi$ in both $r_-$ and $r_+$ but is above this graph in all other points.
Moreover, \cond{G} is equivalent to the \emph{area condition} from \cite{HR10b}, which
characterizes the signed area between the graph of $\Phi^\prime$ and the secant connecting
$r_-$ to $r_+$. With positive and negative sign for above and below the graph of
$\Phi^\prime$, respectively, the area condition reads as follows. The signed area is
non-negative in each stripe $\ccinterval{r_-}{r}$ with $r>r_-$ and non-positive in each
stripe $\ccinterval{r}{r_+}$ with $r<r_+$. We refer to Figure \ref{fig:shocks} for
illustration, where positive and negative area are displayed in dark and light grey colour,
respectively, and recall that the signed area vanishes in the stripe $\ccinterval{r_-}{r_+}$
due to \eqref{Lem:JumpCond.Eqn2}$_1$.
\par
We mention that \cond{G} is truly necessary for the existence of action minimizing fronts,
see Remark \ref{Rem:UnboundedL}. The conditions \cond{M} and \cond{X}, however, are made for
convenience and might be weakened for the price of more technical effort.
\begin{remark} %
\cond{X} is equivalent to
\begin{enumerate}
\item[\cond{S}]
\emph{supersonic front speed}: $\Phi^{\prime\prime}\at{\pm1}<1$,
\end{enumerate}
and \cond{M} implies
\begin{enumerate}
\item[\cond{I}]
\emph{invariant set for $\Phi^{\prime}$}: There exits a constant
$\Gamma>1$ such that $\Phi^\prime$ maps $\ccinterval{-\Gamma}{\Gamma}$ into itself.
\end{enumerate}
\end{remark}
\begin{proof}
\cond{S} follows from the definition of $\Psi$ in \eqref{E:DefPsi}. Towards \cond{I} we
exploit \cond{M} to choose $\tilde{\Gamma}>1$ such that $\Phi^\prime\at{u}>u$ for
$u<-\tilde{\Gamma}$ and $\Phi^\prime\at{u}<u$ for $u>\tilde{\Gamma}$. Then we set
\begin{math}
\Gamma=\max\{\tilde{\Gamma},\,\max_{\abs{u}
\leq%
{\tilde{\Gamma}}}\abs{\Phi^{\prime}\at{u}}\}
\end{math}.
\end{proof}
%
%
%
\subsection{Functionals and operators}%
%
We denote by $\fspaceL^p$, $\fspaceW^{1,\,p}$ and $\fspaceC^k$ the usual function spaces on
the real line, abbreviate the $\fspaceL^p$-norm by $\norm{\cdot}_p$, and write
\begin{align*}
\skp{W_1}{W_2}=\int_\Rset{W_1\at\phase}{W_2}\at\phase\dint\phase
\end{align*}
for the dual pairing of $W_1\in\fspaceL^p$ and $W_2\in\fspaceL^{p'}$ with $1=1/p+1/p'$.
\begin{lemma}
\label{Lem:AProps}
The averaging operator $\calA$ has the following properties:
\begin{enumerate}
\item %
$\calA$ maps $\fspaceL^p$ into $\fspaceL^\infty\cap\fspaceW^{1,\,p}\subset\fspaceC$ for
all $1\leq{p}\leq\infty$ with
\begin{align*}
\at{\calA{W}}^\prime\at\phase
=%
W\at{\phase+\tfrac{1}{2}}-W\at{\phase-\tfrac{1}{2}}
\end{align*} %
and
$\norm{\calA{W}}_{\fspaceL^p}\leq\norm{W}_{\fspaceL^p}$,
$\norm{\calA{W}}_{\fspaceL^\infty}\leq\norm{W}_{\fspaceL^p}$,
$\norm{\calA{W}}_{\fspaceW^{1,\,p}}\leq3\norm{W}_{\fspaceL^p}$.
\item
$\calA$ is symmetric in the sense that $\skp{\calA{W_1}}{W_2}=\skp{W_1}{\calA{W_2}}$
holds for all $W_1\in\fspaceL^p$ and $W_2\in\fspaceL^{p'}$.
\item %
$\calA$ is self-adjoint in $\fspaceL^2$ with spectrum
\begin{math}
\mathrm{spec}\calA=\{\varrho\at{k}\;:\;k\in\Rset\}
\end{math}
where $\varrho\at{k}=\tfrac{2}{k}\sin\at{\tfrac{k}{2}}$.
\end{enumerate}
\end{lemma}
\begin{proof}
The first two statements are straight forward. The third one follows since $\calA$
diagonalizes in Fourier space via
$\calA\mhexp{\iu{k}\phase}=\varrho\at{k}\mhexp{\iu{k}\phase}$.
\end{proof}
We now introduce the affine space
\begin{align*}
\calH=\left\{W\;:\;W-W_\sh\in\fspaceL^2\right\},
\end{align*}
where $W_\sh$ is the shock profile from $\eqref{E:DefShock}$. Exploiting Lemma
\ref{Lem:AProps}, the Taylor expansion of $\Phi^\prime$ around $\pm1$, and the properties of
$W_\sh$ we then find
\begin{align}
\label{Rem:HProps.Eqn4}
\calA{W},\;\Phi^\prime\at{\calA{W}},\;\calA\Phi^\prime\at{\calA{W}}\in\calH,\qquad
W-\calA{W},\;W-\calA^2{W}\in\fspaceL^2.
\end{align}
for all $W\in\calH$. In view of the action integral \eqref{E:DefAction2} we also define a
functional $\calM$ on $\fspaceL^2$ by
\begin{align*}
\calM\nat{V}=
\tfrac{1}{2}\int\limits_\Rset%
V^2-\nat{\calA{V}}^2
\dint\phase=
\tfrac{1}{2}\int\limits_\Rset%
\nat{V-\calA^2V}V
\dint\phase,
\end{align*}
and a functional $\calN$ on $\calH$ by
\begin{align}
\label{Eqn:DefM}
\calN\at{W}=\calM\at{W-W_\sh}+
\tfrac{1}{2}\int\limits_\Rset%
W_\sh^2-\at{\calA{W_\sh}}^2\dint\phase+
\int\limits_\Rset
\at{W-W_\sh}\at{W_\sh-\calA^2{W_\sh}}\dint\phase.
\end{align}
Notice that $\calN\at{W}$ is well defined on $\calH$ as both $W_\sh^2-\at{\calA{W_\sh}}^2$
and $W_\sh-\calA^2{W_\sh}$ have compact support. Moreover, if $W-W_\sh$ decays sufficiently
fast for $\phase\to\pm\infty$ (say $W-W_\sh\in\fspaceL^1$), then we have
\begin{align}
\label{Eqn:MFormula}
\calN\at{W}=
\tfrac{1}{2}\int\limits_\Rset%
W^2-\at{\calA{W}}^2\dint\phase=\tfrac{1}{2}\int\limits_\Rset%
\at{W-\calA^2W}W\dint\phase.
\end{align}
\begin{lemma}
\label{Lem:MProps}
The functional $\calM$ is non-negative and weakly lower semi-continuous on $\fspaceL^2$.
\end{lemma}
\begin{proof}
Denoting the Fourier transform of $V$ by $\wh{V}$ we find
\begin{align}
\notag
\calM\at{V}
&=%
\int\limits_\Rset\nat{1-\varrho\at{k}^2}\wh{V}\at{k}^2\dint{k}
=%
\norm{\sqrt{1-\varrho^2}\,\wh{V}}_2
\end{align}
with $\varrho$ as in Lemma \ref{Lem:AProps}. This gives the desired result as
$V_n\rightharpoonup{V_\infty}$  implies $\wh{V}_n\rightharpoonup{\wh{V}_\infty}$ and hence
$\sqrt{1-\varrho^2}\,\wh{V}_n\rightharpoonup{\sqrt{1-\varrho^2}\,\wh{V}_\infty}$
\end{proof}
\begin{lemma}
\label{Lem:NProps}
The functional $\calN$ is G\^{a}teaux differentiable on $\calH$ with derivative
\begin{align}
\label{Lem:NProps.Eqn4}
\partial\calN\at{W}=W-\calA^2{W}\in\fspaceL^2.
\end{align}
Moreover, $\calN$ is invariant under shifts in $\phase$-direction, and satisfies
\begin{align}
\label{Lem:NProps.Eqn6}
\calN\at{W_2}=\calN\at{W_1}+\calM\at{W_2-W_1}+\skp{W_2-W_1}{W_1-\calA^2{W}_1}
\end{align}
for all $W_1,\,W_2\in\calH$.
\end{lemma}
\begin{proof}
A direct computation with $W\in\calH$ and $\delta{W}\in\fspaceL^2$ shows
\begin{align*}
\skp{\partial\calN\at{W}}{\delta{W}}
&=%
\bat{\skp{W-W_\sh}{\delta{W}}
-\skp{\calA{W}-\calA{W_\sh}}{\calA{\delta{W}}}}
+\skp{W_\sh-\calA^2{W_\sh}}{\delta{W}}
\\&=%
\skp{W-W_\sh}{\delta{W}}-\skp{\calA^2{W}-\calA^2{W_\sh}}{\delta{W}}
+\skp{W_\sh-\calA^2{W_\sh}}{\delta{W}}
\\&=%
\skp{W-\calA^2}{\delta{W}},
\end{align*}
and this gives \eqref{Lem:NProps.Eqn4}. Towards the shift invariance we approximate $W$ by
\begin{align*}
W_n=\chi_{\ccinterval{-n}{n}}\at{W-W_\sh}+W_\sh
\end{align*}
where $\chi_{\ccinterval{-n}{n}}$ is the indicator function of the interval
$\ccinterval{-n}{n}$. Then we use \eqref{Eqn:MFormula} for $W_n$ to find
$\calL\at{W_n}=\calL\at{W_n\at{\cdot+\bar\phase}}$ for all shifts $\bar\phase$, and passing
to the limit $n\to\infty$ gives the desired result. Finally, by definition we have
\begin{align*}
\calN\at{W_2}-\calN\at{W_1}
&=%
\calM\at{W_2-W_\sh}-\calM\at{W_1-W_\sh}+\skp{W_2-W_1}{W_\sh-\calA^2W_\sh}
\end{align*}
and
\begin{align*}
\calM\at{W_2-W_\sh}-\calM\at{W_1-W_\sh}=
\calM\at{W_2-W_1}+
\skp{W_2-W_1}{W_1-W_\sh-\calA^2{W_1}+\calA^2{W_\sh}},
\end{align*}
so \eqref{Lem:NProps.Eqn6} follows from adding both identities.
\end{proof}
To conclude this section we consider the functional
\begin{align*}
\calP\at{W}=%
\int\limits_\Rset{}\Psi\at{\calA{W}}\dint\phase,
\end{align*}
which gives the non-quadratic part of the action integral \eqref{E:DefAction2}
\begin{lemma}
\label{Lem:PProps}
$\calP$ is well defined on $\calH$ with
\begin{align}
\label{Eqn:LowerEstimateForG}
\ul{c}\norm{\calA{W}-\sgn\at{\calA{W}}}_2
\leq
\calP\at{W}
\leq
\ol{c}\norm{\calA{W}-\sgn\at{\calA{W}}}_2
\end{align}
for some constants $\ul{c}$ and $\ol{c}$ that depend only on $\norm{AW}_\infty$. Moreover,
$\calP$ is G\^{a}teaux differentiable on $\calH$ with derivative
$\partial\calP\at{W}=\calA^2{W}-\calA\Phi^\prime\at{\calA{W}}$.
\end{lemma}
\begin{proof}
Let $W\in\calH$ be given and recall that $U=\calA{W}$ satisfies
$U\in\calH\cap\fspaceL^\infty$ due to Lemma \ref{Lem:AProps}. Condition \cond{X} provides two
constants $\ul{c}$ and $\ol{c}$ such that
\begin{align*}
\ul{c}\at{u-\sgn{u}}^2\leq\Psi\at{u}\leq\ol{c}\at{u-\sgn{u}}^2\quad\text\and\quad
\end{align*}
holds for all $\abs{u}\leq{\norm{U}_\infty}$, and we conclude that $\calP$ is well defined
and satisfies \eqref{Eqn:LowerEstimateForG}. Finally, \eqref{Rem:HProps.Eqn4} provides
$\calA^2{W}-\calA\Phi^\prime\at{\calA{W}}\in\fspaceL^2$, so both the existence of and the
formula for $\partial\calP$ follow from a direct calculation.
\end{proof}
%
\subsection{Variational setting}%
%
We now introduce the action functional on $\calH$ by
\begin{align}
\label{Eqn:DefL}%
\calL\at{W}=\calN\at{W}+\calP\at{W}.
\end{align}
In virtue of Lemma \ref{Lem:NProps} and Lemma \ref{Lem:PProps} the functional $\calL$ is well
defined, shift invariant, and G\^{a}teaux differentiable with derivative
\begin{align*}
\partial\calL\at{W}=\calA^2W-\calA\Phi^\prime\at{\calA{W}}\in\fspaceL^2,
\end{align*}
and we conclude that each minimizer of $\calL$ in $\calH$ must solve the front equation
\eqref{Norm.Problem.Eqn}.
However, proving the existence of minimizers in $\calH$ turns out to be difficult and
therefore we restrict $\calL$ to the convex subset
\begin{align*}
\calC=\left\{W\in\calH\cap\fspaceW^{1,\infty}%
\;:\;%
\norm{W}_{\infty}\leq{\Ga}
,\quad%
\norm{W^\prime}_{\infty}\leq2{\Ga}\right\}.
\end{align*}
Notice that the ansatz $W\in\calC$ is reasonable due to condition $\cond{I}$ and since the
front equation \eqref{Norm.Problem.Eqn} combined with Lemma \ref{Lem:AProps} implies
$W\in\calH\cap\fspaceW^{1,\,\infty}$.
\bigpar
In order to link fronts to minimizers of $\calL$ in $\calC$ we observe that the properties of
$\Phi^\prime$ and $\calA$ guarantee $\calC$ to be invariant under the $\fspaceL^2$-gradient
flow of $\calL$. To see this we consider the explicit Euler scheme
\begin{align}
\label{Eqn:EulerScheme}
W\mapsto\calT_\la\at{W}=W-\la\partial\calL\at{W}=\at{1-\la}W+\la\calA\Phi^\prime\at{\calA{W}}
\end{align}
with small step size $\la$.
\begin{lemma}
\label{Lem:FrontsAndMinimisers} The set $\calC$ is invariant under the action of $\calT_\la$
for $0<\la<1$. Consequently, each minimizer of $\calL$ in $\calC$ solves the front
equation \eqref{Norm.Problem.Eqn}.
\end{lemma}
\begin{proof}
For $W\in\calC$ let $P=\Phi^\prime\at{\calA{W}}$ and recall that
$\calA{W},\;P,\;\calA{P}\in\calH$ according to \eqref{Rem:HProps.Eqn4}. Combining \cond{I}
with $\norm{\calA{W}}_\infty\leq\Ga$ and
$\at{\calA{P}}^\prime=P\at{\cdot+1/2}+P\at{\cdot+1/2}$ gives
\begin{align*}
\norm{P}_\infty,\;\norm{\calA{P}}_\infty\leq{\Ga}
,\qquad\norm{\at{\calA{P}}^\prime}_\infty\leq2{\Ga},
\end{align*}
and hence $\calA{P}\in\calC$. Since $\calC$ is convex we also have $\calT_\la\at{W}\in\calC$
for all $0<\la<1$, and passing to the limit $\la\to0$ we then establish the invariance of
$\calC$ under the $\fspaceL^2$-gradient flow of $\calL$. In particular, each minimizer of
$\calL$ in $\calC$ must be a stationary point for the gradient flow of $\calL$ and hence a
solution to the front equation.
\end{proof}
To complete the existence proof for fronts it remains to show that $\calL$ attains its
minimum in $\calC$. We prove this in the next section by using the direct approach, that
means we construct minimizers as limits of minimizing sequences.
\par
A particular problem we have to overcome in the subsequent analysis is that $\calL$ is
\emph{not} coercive on $\calH$. In fact, as illustrated in Figure \ref{fig:sep_phase} there
exist sequences $\at{W_n}_n\subset\calC$ with extending plateaus at $-1$ or $+1$. These
plateaus contribute neither to $\calN$ nor $\calP$ but may imply
\begin{align*}
\norm{W_n-W_\sh\at{\phase_n+\cdot}}_2\xrightarrow{n\to\infty}\infty
\end{align*}
for all choices of the relative shifts $\phase_n$. Heuristically it is clear that the cartoon
from Figure \ref{fig:sep_phase} cannot be prototypical for action minimizing sequences, but
in order to proof this we need a better understanding of sequences with bounded action.

\begin{figure}[ht!]
\centering{%
\includegraphics[width=0.7\textwidth, draft=\figdraft]%
{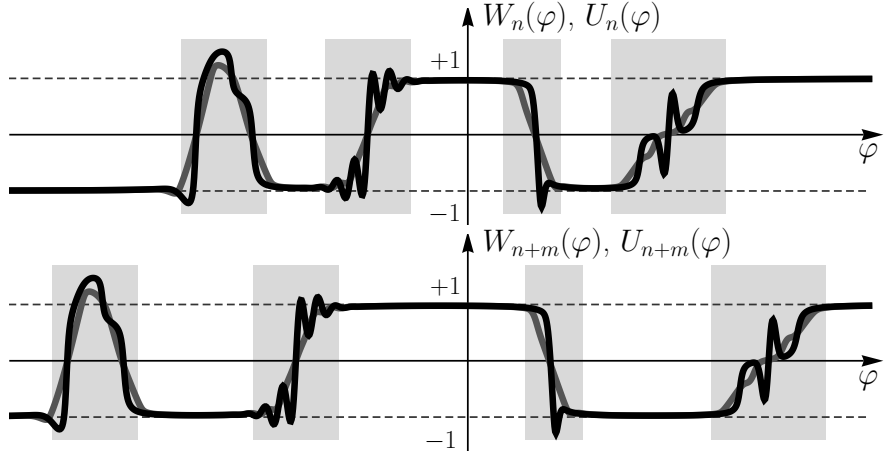}%
\caption{%
Sketch of a sequence with bounded action and two extending plateaus at $\pm1$: graphs of
$W_n$ and $U_n=\calA{W_n}$ in Black in Gray, respectively. The shaded areas indicate
corresponding separations of phases with $I_{n,j}=\phase_{n,j}+J_j$ as in Lemma
\ref{Lem:SepPhasesForSequ}.
}%
\label{fig:sep_phase}
}%
\end{figure}
We conclude with a remark about the necessity of the graph condition \cond{G} and refer to
\S\ref{sec:num} for numerical examples.
\begin{remark}
\label{Rem:UnboundedL} %
Suppose that $\Psi$ satisfies $\cond{M}$, $\cond{X}$, and $\Psi\at{\pm1}=0$, but violates
\cond{G} because there is some $u_\ast$ with $\abs{u_\ast}<\Ga$ and $\Psi\at{u_\ast}<0$. Then
$\calL$ is unbounded from below.
\end{remark}
\begin{proof}
We define a sequence $\at{W_n}_n\subset\calH$ of piecewise linear profiles by
$W_n\at{\phase}=u_\ast$ for $\abs{\phase}\leq{n}$ and $W_n\at{\phase}=\sgn\phase$ for
$\abs{\phase}\geq{n+1/\Ga}$. By construction, $W_n-\calA^2{W}_n$ is supported in
$I_n\cup\at{-I_n}$ with $I_n=\ccinterval{n-1}{n+1+\Ga}$, and a direct calculation shows
\begin{align*}
\calL\at{W_n}\leq{C}+2\at{n-1}\Psi\at{u_\ast}\xrightarrow{n\to\infty}-\infty,
\end{align*}
where $C$ is some constant independent of $n$.
\end{proof}%
%
%
\subsection{Separation of phases for sequences with bounded $\calP$}\label{sec:proof:sop}%
%
To characterize the qualitative properties of a profile $U=\calA{W}$ with $W\in\calC$ we
interpret $U<\tfrac{1}{2}$ and $U>\tfrac{1}{2}$ as \emph{negative phase} and \emph{positive
phase}, respectively, and regard intervals in which $U$ takes intermediate values as
\emph{transition layers}. Obviously, adjacent plateaus of different height are separated by a
transition layer and each $U$ must exhibit at least one transition layer as it connects $-1$
to $+1$.
\par%
We next exploit the uniform $\fspaceL^{\infty}$-bound for $W\in\calC$ to derive a lower
bound for the $\calP$-contribution of each transition layer. To this end we introduce
\begin{align*}
Z_U=\{\bar\phase\;:\;\abs{U\at{\bar\phase}}\leq\tfrac{1}{2}\}.
\end{align*}
which is nonempty, closed and bounded as $U\in\calC$ is continuous with $U\at\phase\to\pm1$
as $\phase\to\pm\infty$.
\begin{remark}
\label{Rem:MinCostOfZero}
There exist constants $\bar{\eta}>0$ and $\bar{\mu}>0$ such that
\begin{align*}
\int\limits_{\bar\phase-\bar\eta}^{\bar\phase+\bar\eta}
\Psi\at{U\at\phase}\dint\phase
>%
\bar\mu
\end{align*}
for each $U\in\calC$ and $\bar{\phase}\in{Z}_U$.
\end{remark}
\begin{proof}
This follows since $U\in\calC$ implies
\begin{math}
\abs{U\at{\phase_2}-U\at{\phase_1}}\leq%
\int_{\phase_2}^{\phase_1}\abs{U^\prime\at\phase}\dint\phase
\leq%
2\Gamma\abs{\phase_2-\phase_1}
\end{math}. %
In particular, we have $\abs{U\at{\phase}}\leq3/4$ for all
$\abs{\phase-\bar\phase}\leq\bar\eta=1/\at{8\Ga}$ and the claim follows with
$\bar\mu=\bar\sigma/\at{2\bar\eta}$ and $\bar\sigma=\sup_{\abs{u}\leq3/4}\Psi\at{u}>0$.
\end{proof}
In order to show that each function $U$ possesses a finite number of transition layers we
introduce the following definition. A \emph{separation of phases} for a given profile
$U\in\calC$ is a finite collection of closed intervals (transition layers)
$I_1,\,\tdots,\,I_m$, $m\geq1$, such that
\begin{enumerate}
\item
the intervals are disjoint and ordered, i.e.,
\begin{align*}
\min{I_{1}}<\max{I_{1}}<\min{I_2}<\tdots<%
\max{I_{m-1}}<\min{I_{m}}<\max{I_{m}},
\end{align*}
\item
$Z_U$ is contained in
$I_1\cup\tdots\cup{I_m}$.
\item
for each interval $I_j$ there exists $\bar\phase_j\in{Z_U}$ such that with
$\ccinterval{\bar\phase_j-\bar\eta}{\bar\phase_j+\bar\eta}\subset{I_j}$,
\end{enumerate}
\begin{lemma}
\label{Lem:ExSepPhases}
Let $W\in\calC$ be given and set $U=\calA{W}$. Then there exists a separation of phases
$I_1,\,\tdots,\,I_m$ for $U$ with
\begin{align*}
m\leq\mathrm{floor}\at{\calP\at{W}/\bar\mu}\quad\text{and}\quad
2\bar\eta\leq\abs{I_j}\leq4\calP\at{W}\bar{\eta}/\bar\mu
\end{align*}
for all $j=1\tdots{m}$.
\end{lemma}
\begin{proof}
We define a finite number of points $\bar{\phase}_j\in{Z_U}$ and intervals
$I_j=\oointerval{\bar{\phase}_j-2\bar\eta}{\bar{\phase}_j-2\bar\eta}$ iteratively as follows:
$\bar{\phase}_1$ is the smallest element of $Z_U$, i.e. $U\at\phase<-\tfrac{1}{2}$ for all
$\phase<\bar{\phase}_1$. If $Z_U\setminus{I_1}$ is empty we stop the iteration; otherwise we
choose $\bar{\phase}_2$ to be the smallest zero of $Z_U$ outside of $I_1$. Then we have
$\bar{\phase}_2-\bar{\phase}_1\geq2\bar\eta$, so Remark \ref{Rem:MinCostOfZero} yields
\begin{align*}
\calP\at{W}\geq
\int\limits_{\bar{\phase}_1-\bar\eta}^{\bar{\phase}_1+\bar\eta}
\Psi\at{U\at\phase}\dint\phase+
\int\limits_{\bar{\phase}_2-\bar\eta}^{\bar{\phase}_2+\bar\eta}
\Psi\at{U\at\phase}\dint\phase
\geq{2}\bar\mu.
\end{align*}
If possible, we now define $\bar{\phase}_3$ as the minimum of $Z_U\setminus\at{I_1\cup{I_2}}$
and proceed iteratively until the iteration stops after
$m\leq\mathrm{floor}\at{\calP\at{W}/\bar\mu}$ steps. By construction, we have
$Z_U\subset\bigcup_{j=1}^m{I_j}$ and $\abs{I_j}\leq4\bar\eta$ for all $j$ but the intervals
$I_j$ may overlap. Finally, we obtain the desired separation of phases by merging overlapping
intervals.
\end{proof}
Our main result in this section concerns sequences $\at{W_n}_n\in\calC$ with bounded
$\calP\at{W_n}$. It guarantees, roughly speaking, the existence of \emph{compatible}
transitions layers which $\at{i}$ have the same length, $\at{ii}$ separate the same phases,
and $\at{iii}$ depart from each other. For an illustration we refer to Figure
\ref{fig:sep_phase}.
\begin{lemma}
\label{Lem:SepPhasesForSequ}
Let $\at{W_n}_{n}\subset\calC$ be a sequence with $\limsup_{n\to\infty}\calP\at{W_n}<\infty$.
Then there exists a not relabelled subsequence along with a finite number of intervals
$J_1,\,\tdots,\,J_m$ with the following properties:
\begin{enumerate}
\item
Each interval $J_j$ is centred around zero, i.e., $J_j=\ccinterval{-\eta_j}{\eta_j}$ for
some $0<\eta_j<\infty$.
\item
For each $n$ there exist shifts $\phase_{n,1}<\tdots<\phase_{n,m}$ such that
\begin{enumerate}
\item
$\phase_{n,1}+J_1,\,\tdots,\,\phase_{n,m}+J_m$ is a separation of phases for
$U_n=\calA{W_n}$,
\item
$\phase_{n,j+1}-\phase_{n,j}\to\infty$ as $n\to\infty$ for all
$j=1{\tdots}m-1$.
\end{enumerate}
\item
There exists a choice of signs $s_0,\,\tdots,\,s_{m}$ with $s_0=-1$, $s_m=1$, and
$s_j\in\{-1,\,+1\}$ such that
\begin{align*}
\begin{array}{lclclclcl}
\sgn\at{U_n\at\phase}&=&s_0
&\quad\text{for}\quad&%
&&\phase&<&\phase_{n,1}-\eta_{1}
\\%
\sgn\at{U_n\at\phase}&=&s_1
&\quad\text{for}\quad&%
\phase_{n,1}+\eta_1&<&\phase&<&\phase_{n,2}-\eta_{2}
\\%
&&&\tdots&&&&&
\\%
\sgn\at{U_n\at\phase}&=&s_{m-1}
&\quad\text{for}\quad&%
\phase_{n,m-1}+\eta_{m-1}&<&\phase&<&\phase_{n,m}-\eta_m
\\%
\sgn\at{U_n\at\phase}&=&s_{m+1}
&\quad\text{for}\quad&%
\phase_{n,m}+\eta_m&<&\phase&&
\end{array}
\end{align*}
hold for all $n$.
\end{enumerate}
\end{lemma}
\begin{proof}
Thanks to Lemma \ref{Lem:ExSepPhases} we can extract a subsequence such that each $U_n$ has a
separation of phases that consists of $m$ intervals $I_{n,1},\tdots,{I_{n,m}}$ where $m$ is
independent of $n$. We  denote the centre of $I_{n,j}$ by $\phase_{n,j}$, set
$J_{n,j}=I_{n,j}-\phase_{n,j}$, and notice that $\ul{c}\leq\abs{J_{n,j}}\leq\ol{c}$ and
$\phase_{n,j+1}-\phase_{n,j}>\ul{c}$ for some constants $\ol{c}>\ul{c}>0$ independent of $n$
and $j$. In particular, the intervals $J_{j}=\bigcup_{n}J_{n,j}$ have finite length
$\ol{c}\leq\abs{J_j}\leq\ol{c}$.
\par
Our strategy for the proof is to refine $m$, the subsequence $\at{W_n}_n$, the phase shifts
$\phase_{n,j}$, and the intervals $J_j$ in several steps. To this end we start the following
algorithm at level $1$.
\begin{enumerate}
\item[]%
Level k
\begin{enumerate}
\item[]
If $m=k$, then we stop the algorithm.
\item[]
If $\phase_{n,k+1}-\phase_{n,k}\to\infty$ as $n\to\infty$ along a subsequence, then
we extract this subsequence and jump to level $k+1$.
\item[] %
If $m<k$ and $0<\sup_{n}\at{\phase_{n,k+1}-\phase_{n,k}}<\infty$, then we merge $J_k$
and $J_{k+1}$ as follows: At first we choose $\tilde{J}_{k}$ sufficiently large such
that
\begin{align*}
\tilde{J}_{k}\supseteq{J}_{k}\cup%
{\bigcup_{n}\at{\phase_{n,k+1}-\phase_{n,k}+{J}_{k+1}}}%
\end{align*}
Secondly we define $\tilde{\phase}_{n,k}=\phase_{n,k}$,
$\tilde{m}=m-1$ and
\begin{align*}
\begin{array}{lclclclcll}
\tilde{J}_j&=&J_j&\quad\text{and}\quad&\tilde{\phase}_{n,j}
&=&\phase_{n,j}&\quad\text{for}\quad&{j<k},\\
\tilde{J}_j&=&J_{j+1}&\quad\text{and}\quad&\tilde{\phase}_{n,j}&=&
\phase_{n,j+1}&\quad\text{for}\quad&{j>k}.
\end{array}
\end{align*}
Finally we restart level $k$ with $\tilde{m}$,
$\tilde{\phase}_{n,j}$, and $\tilde{J}_j$ instead of  $m$,
$\phase_{n,j}$, and $J_j$.
\end{enumerate}
\end{enumerate}
This algorithm stops after a finite number of steps when $m-k=0$. It provides intervals $J_j$
and phase shifts $\phase_{n,j}$ for $j=1\tdots{m}$ and $n\in{N}$ with
$\lim_{n\to\infty}\phase_{n,j+1}-\phase_{n,j}=\infty$ for all $j$. By extracting subsequences
we can also ensure that, for each $n$, the intervals $\phase_{n_j}+J_{n,j}$ are pairwise
disjoint and provide therefore a separation of phases for $U_n$. Finally, by extracting
further subsequences if necessary we guarantee the existence of a choice of signs.
\end{proof}
%
%
\subsection{Existence of minimizers for $\calL$}%
%
We now finish the existence proof for fronts.
\begin{theorem}
\label{Theo:Minimiser}%
$\calL$ attains its minimum on $\calC$ and each minimizer is a front.
\end{theorem}
\begin{proof}
\emph{\underline{Step 0.}} %
We start with some notations. For a given minimizing sequence  $\at{W_n}_n\subset\calC$ we
define
\begin{align*}
U_n=\calA{W_n},\qquad{S_n}=\sgn{U_n},
\end{align*}
and for each $K>0$ we introduce the operator
\begin{align*}
E_K:\calC\to\calC,\quad E_KW=\chi_{\ccinterval{-K}{K}}\at{W-W_\sh}+W_\sh.
\end{align*}
Here $\chi_{\ccinterval{-K}{K}}$ denotes the usual indicator function, so we have
$\norm{E_KW-W}_2\to0$ as $K\to\infty$ for each $W\in\calC$. Finally, within this proof $C$
always denotes a positive constant that is independent of $n$ and $K$, but the value of $C$
may change from line to line.
\par
\emph{\underline{Step 1.}}
By assumption and $\calM\at{W_n-W_\sh}\geq0$ we have
\begin{align}
\label{Theo:Minimiser.ZEqn1}
\calP\at{W_n}
=%
{\calL\at{W_n}}-{\calN\at{W_n}}
\leq%
{C}+\norm{W-W_\sh}_{\infty}\norm{W_\sh-\calA^2W_\sh}_1
\leq{C}.
\end{align}
Therefore we can extract (a not relabelled) subsequence for which Lemma
\ref{Lem:SepPhasesForSequ} provides a finite number of intervals $J_1\tdots{J}_m$, sequences
of phase shifts $\at{\phase_{n,\,j}}_n$ and a choice of signs $\at{s_{n,\,j}}_n$. There exits
at least one $1\leq{j_\ast}\leq{m}$ such that $s_{j_\ast-1}=-1$ and $s_{j_\ast}=+1$, and
since $\calL$ is invariant under shifts we can assume that $\phase_{n,\,j_\ast}=0$. With
$J_{j_\ast}=\ccinterval{-\eta_{j_\ast}}{\eta_{j_\ast}}$ and due to
$\lim_{n\to\infty}\phase_{n,\,j+1}-\phase_{n,\,j}=\infty$ we then have
\begin{align*}
\limsup\limits_{n\to\infty}
U_{n}\at\phase\leq-\tfrac{1}{2}
\quad\text{for}\quad
\phase<-\eta_{j_\ast},
\qquad
\liminf\limits_{n\to\infty}U_{n}\at\phase\geq\tfrac{1}{2}
\quad\text{for}\quad
\phase>\eta_{j_\ast}.
\end{align*}
By compactness we can extract a  further subsequence such that $W_n\rightharpoonup{W_\infty}$
weakly$\star$ in $\fspaceW^{1,\,\infty}$. In particular, $W_n$ converges to $W_\infty$
uniformly on each compact interval, and hence
\begin{align}
\label{Theo:Minimiser.EqnA1}%
E_K{W_n}\xrightarrow{n\to\infty}E_KW_\infty
,\quad%
E_K{U_n}\xrightarrow{n\to\infty}E_KU_\infty
\quad%
\text{strongly in $\fspaceL^2$ for all $K$}.
\end{align}
\par
\emph{\underline{Step 2.}} %
Towards $W_\infty\in\calC$ we show that $W_n-S_n$ is uniformly bounded in $\fspaceL^2$. The
first observation is that \eqref{Theo:Minimiser.ZEqn1} combined with
\eqref{Eqn:LowerEstimateForG} implies
\begin{align}
\label{Theo:Minimiser.EqnB1}
\norm{U_n-S_n}_2\leq{C}.
\end{align}
The second observation is that both $S_n-{\calA}S_n$ and $S_n-\calA^2S_n$ are supported in
the $1$-neighbourhood of $I_n=\bigcup_{j=1}^m\at{J_j+\phase_{n,\,j}}$. Therefore,
$\abs{I_n}\leq\sum_{j=1}^{m}\abs{J_j}\leq{C}$ yields
\begin{align*}
\norm{S_n-{\calA}S_n}_2\leq{C}
,\qquad
\norm{S_n-\calA^2S_n}_1\leq{C}
,\qquad%
\abs{\calN\at{S_n}}=\abs{\skp{S_n-\calA^2S_n}{S_n}}\leq{C},
\end{align*}
and by \eqref{Theo:Minimiser.EqnB1} we find
\begin{align}
\label{Theo:Minimiser.EqnB2}%
\norm{U_n-{\calA}S_n}_2\leq\norm{U_n-S_n}_2+\norm{S_n-{\calA}S_n}_2\leq{C}.
\end{align}
Exploiting \eqref{Lem:NProps.Eqn6} for $W_2=W_n$ and $W_1=S_n$ gives
\begin{align*}
\calM\at{W_n-S_n}
&=%
{\calN\at{W_n}}-\calN\at{S_n}-\skp{W_n-S_n}{S_n-\calA^2S_n}
\\&\leq%
\calL\at{W_n}+\abs{\calN\at{S_n}}+
\at{\norm{W_n}_\infty+\norm{S_n}_\infty}\norm{S_n-\calA^2S_n}_1\leq{C},
\end{align*}
and with \eqref{Theo:Minimiser.EqnB2} we obtain
\begin{align}
\label{Theo:Minimiser.EqnB3}%
\norm{W_n-S_n}^2_2&=\calM\at{W_n-S_n}+\norm{U_n-{\calA}S_n}^2_2
\leq{C}.
\end{align}
Now we are able to show $W_\infty\in\calC$. From \eqref{Theo:Minimiser.EqnB3} we infer that
\begin{align*}
\norm{E_KW_n-W_\sh}_2
\leq\norm{E_KW_n-E_KS_n}_2+\norm{E_K{S_n}-E_KW_\sh}_2\leq{C}+\norm{E_K{S_n}-E_KW_\sh}_2,
\end{align*}
and with $S_n\at{\phase}\to{W_\sh}\at\phase$ as $n\to\infty$ for all $\phase\neq{J}_{j_\ast}$
we find
\begin{align*}
\norm{E_KW_\infty-W_\sh}_2\leq{C}.
\end{align*}
Passing to the limit $K\to\infty$ now gives $W\in\calH$, and $W\in\calC$ follows because
$W_\infty$ was defined as weak$\star$ limit in $\fspaceW^{1,\,\infty}$.
\par
\emph{\underline{Step 3.}} %
There remains to show that $W_\infty$ minimizes $\calL$. From \eqref{Lem:NProps.Eqn6} we
infer that
\begin{align}
\label{Theo:Minimiser.EqnC1}
\calN\at{W_n}
=%
\calN\at{E_KW_n}+\calM\at{W_n-E_KW_n}+
\skp{W_n-E_KW_n}{E_KW_n-\calA^2\at{E_KW_n}}
\end{align}
and
\begin{align}
\label{Theo:Minimiser.EqnCa}
\begin{split}
\abs{\calN\at{E_K{W}_n}-\calN\at{E_K{W}_\infty}}
\leq&\quad%
\calM\at{E_KW_n-E_KW_\infty}\\&+
\norm{E_KW_n-E_KW_\infty}_2\norm{E_KW_\infty-\calA^2\at{E_KW_\infty}}_2
\end{split}
\end{align}
hold for all $K$ and $n\in\Nset\cup\{\infty\}$. Combining \eqref{Theo:Minimiser.EqnCa} with
\eqref{Theo:Minimiser.EqnA1} gives
\begin{align*}
\calN\at{E_K{W}_n}\xrightarrow{n\to\infty}\calN\at{E_K{W}_\infty}.
\end{align*}
Moreover, since $E_K{W_n}-\calA^2\at{E_K{W_n}}$ is supported in $\ccinterval{-K-1}{K+1}$ for
all $n$, we also have
\begin{align*}
\skp{W_n-E_KW_n}{E_KW_n-\calA^2\at{E_KW_n}}
\xrightarrow{n\to\infty}%
\skp{W_\infty-E_KW_\infty}{E_KW_\infty-\calA^2\at{E_KW_\infty}},
\end{align*}
and passing to the limit $n\to\infty$ in \eqref{Theo:Minimiser.EqnC1} provides
\begin{align}
\label{Theo:Minimiser.EqnC2}
\liminf\limits_{n\to\infty}\calN\at{W_n}
\geq%
\calN\at{E_KW_\infty}+
\skp{W_\infty-E_KW_\infty}{E_KW_\infty-\calA^2\at{E_KW_\infty}},
\end{align}
where we used that $\calM\at{W_n-E_KW_n}\geq0$ according to Lemma \ref{Lem:MProps}. On the
other hand, evaluating \eqref{Theo:Minimiser.EqnC1} for $n=\infty$ gives
\begin{align*}
\calN\at{W_\infty}=\calN\at{E_KW_\infty}+\calM\at{W_\infty-E_KW_\infty}+
\skp{W_\infty-E_KW_\infty}{E_KW_\infty-\calA^2\at{E_KW_\infty}}
\end{align*}
and due to $\calM\at{W_\infty-E_KW_\infty}\to0$ as $K\to\infty$ we find
\begin{align}
\label{Theo:Minimiser.EqnC3}
\calN\at{E_KW_\infty}+
\skp{W_\infty-E_KW_\infty}{E_KW_\infty-\calA^2\at{E_KW_\infty}}
\xrightarrow{K\to\infty}\calN\at{W_\infty}.
\end{align}
The combination of \eqref{Theo:Minimiser.EqnC2} and \eqref{Theo:Minimiser.EqnC3} reveals
\begin{align}
\label{Theo:Minimiser.EqnC4}
\liminf\limits_{n\to\infty}\calN\at{W_n}\geq\calN\at{W_\infty},
\end{align}
and Fatou's Lemma provides
\begin{align}
\label{Theo:Minimiser.EqnC5}
\liminf\limits_{n\to\infty}\calP\at{W_n}\geq\calP\at{W_\infty}
\end{align}
due to $\Psi\geq0$ and since $W_n$ converges to $W_\infty$ pointwise. Adding
\eqref{Theo:Minimiser.EqnC4} and \eqref{Theo:Minimiser.EqnC5} we conclude that $W_\infty$ is
in fact a minimizer of $\calL$, and Lemma \ref{Lem:FrontsAndMinimisers} guarantees that
$W_\infty$ solves the front equation \eqref{Norm.Problem.Eqn}.
\end{proof}
We conclude with some remarks.
\begin{enumerate}
\item
Theorem \ref{Intro:MainTheo} follows by combining Lemma \ref{Lem:JumpCond}, Lemma
\ref{Lem:Normalisation}, Lemma \ref{Lem:FrontsAndMinimisers}, Remark
\ref{Rem:UnboundedL}, and Theorem \ref{Theo:Minimiser}.
\item
The assertions of Theorem \ref{Theo:Minimiser} can be sharpened as follows. For each
minimizing sequence $\at{W_n}_n$ we have equality signs in both
\eqref{Theo:Minimiser.EqnC4} and \eqref{Theo:Minimiser.EqnC5}, and hence
\begin{align*}
\lim\limits_{K\to\infty}\lim\limits_{n\to\infty}\calM\at{W_n-E_KW_n}=0.
\end{align*}
This implies that there is only one interval $J_1=J_{j_\ast}$ and in turn that
$W_n-W_\infty\to0$ strongly in $\fspaceL^2$. In this sense each minimizing sequence obeys
exactly one transition from negative phase to positive phase.
\item
If $\Phi^\prime$ is increasing in $\ccinterval{-1}{+1}$ we can improve the existence
result for fronts as follows. We choose $\Ga=1$ in \cond{I} and consider the set
\begin{align*}
\wt{\calC}=\{W\in\calC\;:\;W^\prime\geq0\}.
\end{align*}
Then $\wt{\calC}$ is an invariant set for the gradient flow of $\calL$ and again one can
show that $\calL$ restricted to $\wt{\calC}$ attains it minimum (a proof tailored to
monotone profiles is given in \cite{HR10b}). In particular, in this case there exist
action minimizing fronts with monotone profile $W$.
\end{enumerate}
%
%
%
\section{Approximation of fronts}\label{sec:num}
%
%
In this section we illustrate the analytical results from \S\ref{sec:proof} by some numerical
simulations. To this end we discretize the Euler scheme for the gradient flow of $\calL$, see \eqref{Eqn:EulerScheme}, as
follows:
\begin{samepage}
\begin{enumerate}
\item
Fix a finite interval $\ccinterval{-\Len}{+\Len}$ and introduce equidistant grid
points by $\phase_k=-\Len+2k\Len/D$, where $k=0\tdots{D}$ and $D\in\Nset$ is large.
\item
Approximate each profile $W\in\calC$ by the discrete vector $W_i=W\at{\phase_i}$ and
impose the boundary conditions $W_i=-1$ and $W_i=+1$ for $i<0$ and $i>D$, respectively.
\item
Replace the integrals in the definition of $\calA$ by Riemann sums with respect to the
$\phase_i$'s.
\item Choose $\la$ sufficiently small and initialize the iteration
    \eqref{Eqn:EulerScheme} with shock initial data $W_i=\sgn{\phase_i}$.
\end{enumerate}
\end{samepage}
\begin{figure}[ht!]
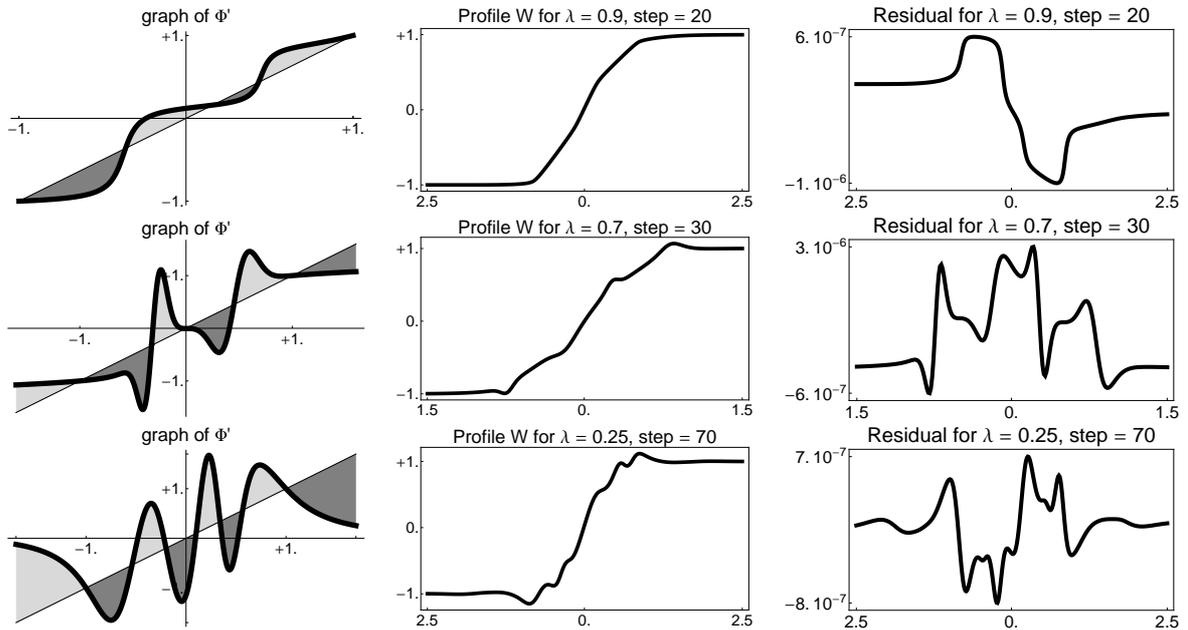
%
\centering{%
\begin{minipage}[c]{0.285\textwidth}%
\includegraphics[width=\textwidth, draft=\figdraft]%
{\figfile{ex_1_force}}%
\end{minipage}%
\hspace{0.025\textwidth}%
\begin{minipage}[c]{0.285\textwidth}%
\includegraphics[width=\textwidth, draft=\figdraft]%
{\figfile{ex_1_prof}}%
\end{minipage}%
\hspace{0.025\textwidth}%
\begin{minipage}[c]{0.315\textwidth}%
\includegraphics[width=\textwidth, draft=\figdraft]%
{\figfile{ex_1_res}}%
\end{minipage}%
\\%
\begin{minipage}[c]{0.285\textwidth}%
\includegraphics[width=\textwidth, draft=\figdraft]%
{\figfile{ex_2_force}}%
\end{minipage}%
\hspace{0.025\textwidth}%
\begin{minipage}[c]{0.285\textwidth}%
\includegraphics[width=\textwidth, draft=\figdraft]%
{\figfile{ex_2_prof}}%
\end{minipage}%
\hspace{0.025\textwidth}%
\begin{minipage}[c]{0.315\textwidth}%
\includegraphics[width=\textwidth, draft=\figdraft]%
{\figfile{ex_2_res}}%
\end{minipage}%
\\%
\begin{minipage}[c]{0.285\textwidth}%
\includegraphics[width=\textwidth, draft=\figdraft]%
{\figfile{ex_3_force}}%
\end{minipage}%
\hspace{0.025\textwidth}%
\begin{minipage}[c]{0.285\textwidth}%
\includegraphics[width=\textwidth, draft=\figdraft]%
{\figfile{ex_3_prof}}%
\end{minipage}%
\hspace{0.025\textwidth}%
\begin{minipage}[c]{0.315\textwidth}%
\includegraphics[width=\textwidth, draft=\figdraft]%
{\figfile{ex_3_res}}%
\end{minipage}%
}%
\caption{%
Three examples with admissible potential as in Assumption \ref{Ass:Pot} where $\Phi^\prime$
is plotted in the invariant interval. During the iteration the profiles $W$ converge to a
front.
}%
\label{Fig:Num1}%
\end{figure}%
\begin{figure}[ht!]
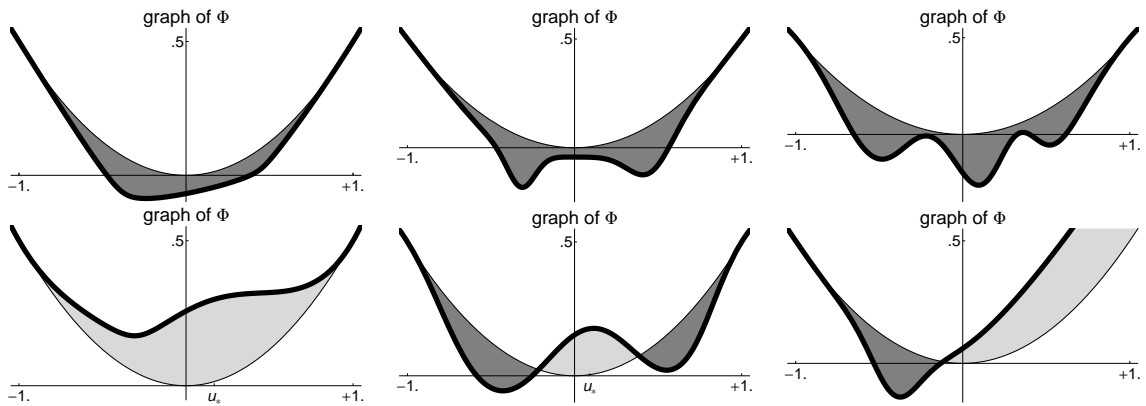
%
\centering{%
\begin{minipage}[c]{0.285\textwidth}%
\includegraphics[width=\textwidth, draft=\figdraft]%
{\figfile{ex_1_pot}}%
\end{minipage}%
\hspace{0.025\textwidth}%
\begin{minipage}[c]{0.285\textwidth}%
\includegraphics[width=\textwidth, draft=\figdraft]%
{\figfile{ex_2_pot}}%
\end{minipage}%
\hspace{0.025\textwidth}%
\begin{minipage}[c]{0.285\textwidth}%
\includegraphics[width=\textwidth, draft=\figdraft]%
{\figfile{ex_3_pot}}%
\end{minipage}%
\\%
\begin{minipage}[c]{0.285\textwidth}%
\includegraphics[width=\textwidth, draft=\figdraft]%
{\figfile{ex_4_pot}}%
\end{minipage}%
\hspace{0.025\textwidth}%
\begin{minipage}[c]{0.285\textwidth}%
\includegraphics[width=\textwidth, draft=\figdraft]%
{\figfile{ex_5_pot}}%
\end{minipage}%
\hspace{0.025\textwidth}%
\begin{minipage}[c]{0.285\textwidth}%
\includegraphics[width=\textwidth, draft=\figdraft]%
{\figfile{ex_6_pot}}%
\end{minipage}%
}%
\caption{%
Potentials $\Phi$ with front parabolas for the examples from Figure \ref{Fig:Num1}
(top row) and \ref{Fig:Num3} (bottom row).
}%
\label{Fig:Num2}%
\end{figure}%
\begin{figure}[ht!]
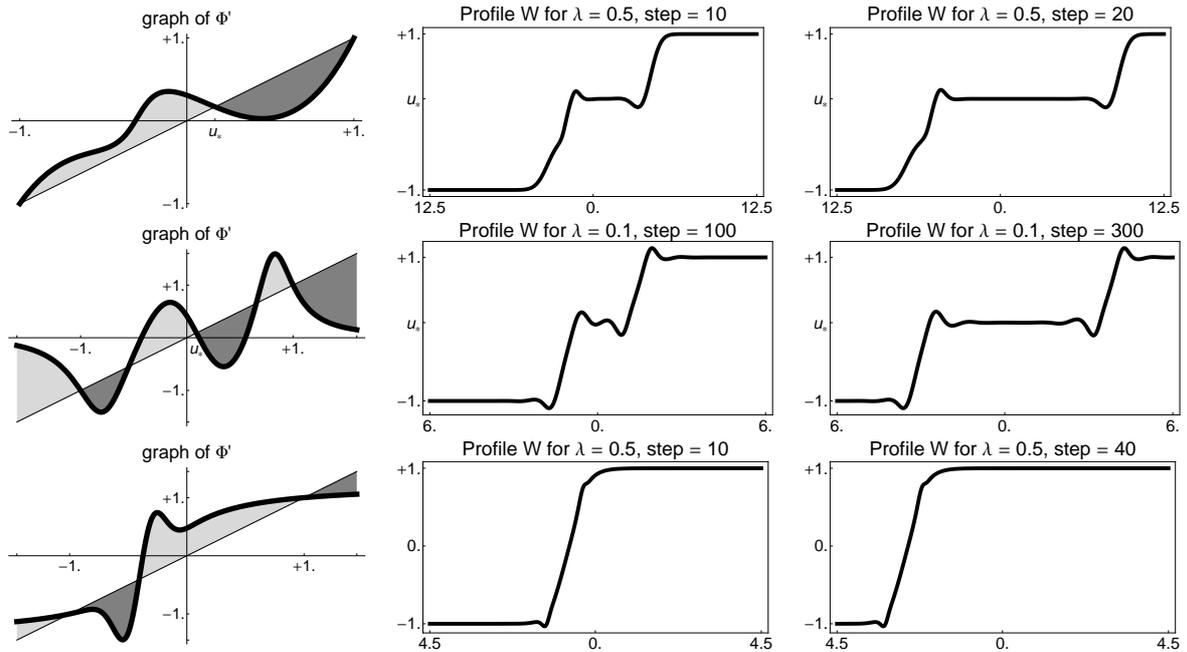
%
\centering{%
\begin{minipage}[c]{0.285\textwidth}%
\includegraphics[width=\textwidth, draft=\figdraft]%
{\figfile{ex_4_force}}%
\end{minipage}%
\hspace{0.025\textwidth}%
\begin{minipage}[c]{0.3\textwidth}%
\includegraphics[width=\textwidth, draft=\figdraft]%
{\figfile{ex_4_prof_1}}%
\end{minipage}%
\hspace{0.025\textwidth}%
\begin{minipage}[c]{0.3\textwidth}%
\includegraphics[width=\textwidth, draft=\figdraft]%
{\figfile{ex_4_prof_2}}%
\end{minipage}%
\\%
\begin{minipage}[c]{0.285\textwidth}%
\includegraphics[width=\textwidth, draft=\figdraft]%
{\figfile{ex_5_force}}%
\end{minipage}%
\hspace{0.025\textwidth}%
\begin{minipage}[c]{0.3\textwidth}%
\includegraphics[width=\textwidth, draft=\figdraft]%
{\figfile{ex_5_prof_1}}%
\end{minipage}%
\hspace{0.025\textwidth}%
\begin{minipage}[c]{0.3\textwidth}%
\includegraphics[width=\textwidth, draft=\figdraft]%
{\figfile{ex_5_prof_2}}%
\end{minipage}%
\\%
\begin{minipage}[c]{0.285\textwidth}%
\includegraphics[width=\textwidth, draft=\figdraft]%
{\figfile{ex_6_force}}%
\end{minipage}%
\hspace{0.025\textwidth}%
\begin{minipage}[c]{0.3\textwidth}%
\includegraphics[width=\textwidth, draft=\figdraft]%
{\figfile{ex_6_prof_1}}%
\end{minipage}%
\hspace{0.025\textwidth}%
\begin{minipage}[c]{0.3\textwidth}%
\includegraphics[width=\textwidth, draft=\figdraft]%
{\figfile{ex_6_prof_2}}%
\end{minipage}%
}%
\caption{%
Three counterexamples for the existence of action minimizing fronts. The first two examples
satisfy the macroscopic constraints but violate the graph condition \cond{G}, so the 
iteration minimizes the action via an extending plateau. The third example violates the 
constraint $\Phi\at{+1}=\Phi\at{-1}$, so
the profiles converge to one of the asymptotic states.
}%
\label{Fig:Num3}%
\end{figure}%
In numerical simulations, the resulting iteration scheme has good convergence properties and
decreases the action provided that $\la$ is sufficiently small and $L$ is sufficiently large.
Three examples with normalized front data and $\Phi$ as in Assumption \ref{Ass:Pot} are shown
in Figure \ref{Fig:Num1}, where $\Phi^\prime$ is always plotted over the invariant interval
$\ccinterval{-\Ga}{\Ga}$. For plots of $\Phi$ see Figure \ref{Fig:Num2}. In the first example
$\Phi^\prime$ is increasing in $\ccinterval{-1}{+1}$ and the front profile $W$ turns out to
be monotone. We refer to the remark at the end of \S\ref{sec:proof} for an explanation, and
to \cite{HR10b} for more examples. The other two examples in Figure \ref{Fig:Num1} illustrate
that the front profiles for non-convex $\Phi$ are in general non-monotone.
\bigpar
In order to illustrate the necessity of the graph condition \cond{G} we present in Figure
\ref{Fig:Num3} simulations for potentials that violate this condition. In the first example
the interval $\ccinterval{-1}{1}$ is invariant under $\Phi^\prime$ but $\Psi$ is negative in
this interval. After some initial iterations the profiles exhibit an extending plateau at
$-1<w_\ast<1$, where $w_\ast$ is the minimizer of $\Psi$ in $\ccinterval{-1}{1}$ and
satisfies $w_\ast=\Phi^\prime\at{w_\ast}$. The onset of the extending plateau is a direct
consequence of the energy landscape of $\calL$, see Remark \ref{Rem:UnboundedL}. The second
simulation provides another counterexample for the existence of action minimizing fronts.
Here $\Psi$ has the correct sign close to $\pm1$ but attains again a negative minimum in
$-1<w_\ast<1$. As before, the profiles minimize their action by converging to the `global
minimizer' $W\equiv{w_\ast}$. Finally, the third example illustrates what happens if the
asymptotic states violate the macroscopic constraints. More precisely, here we violate
$\eqref{Lem:JumpCond.Eqn2}$ due to $\Phi\at{+1}>\Phi\at{-1}$, and observe that the profiles
converge to the global minimizer $W\equiv+1$.
%
%
\section*{Acknowledgements}%
%
%
%
\providecommand{\bysame}{\leavevmode\hbox to3em{\hrulefill}\thinspace}
\providecommand{\MR}{\relax\ifhmode\unskip\space\fi MR }
\providecommand{\MRhref}[2]{%
  \href{http://www.ams.org/mathscinet-getitem?mr=#1}{#2}
}
\providecommand{\href}[2]{#2}

\end{document}